\documentclass[12pt]{article}
\usepackage{amssymb, url}
\usepackage{amsmath,amsthm}
\usepackage{verbatim}
\usepackage{graphicx}
\usepackage{epstopdf}
\usepackage{fullpage}
\usepackage[hang,flushmargin]{footmisc}
\usepackage{hyperref}

\newtheorem{theorem}{Theorem}
\newtheorem{lemma}[theorem]{Lemma}

\newtheorem{prop}[theorem]{Proposition}
\newtheorem{cor}[theorem]{Corollary}

\newcommand{\rpot}{r_{pot}}

\DeclareMathOperator{\comp}{comp}
\newcommand{\st}{\colon\,}
\newcommand{\Z}{\mathbb{Z}}

\title{Chv\'{a}tal-type results for degree sequence Ramsey numbers}

\author{ Christopher Cox$^{1,3,4}$\and Michael Ferrara$^{2,3,4}$ \and Ryan R.\ Martin$^{1,3,6}$ \and Benjamin Reiniger$^{3,4,5}$}

\begin{document}
\maketitle
\footnotetext[1]{Department of Mathematics, Iowa State University, Ames, IA 50011; {\tt $\{$cocox,rymartin$\}$@iastate.edu}}
\footnotetext[2]{Department of Mathematical and Statistical Sciences, University of Colorado Denver, Denver, CO 80217 ; {\tt michael.ferrara@ucdenver.edu}.}
\footnotetext[3]{Research supported in part by NSF grant DMS-1427526, ``The Rocky Mountain - Great Plains Graduate Research Workshop in Combinatorics".}  
\footnotetext[4]{Research supported in part by Simons Foundation Collaboration Grant \#206692 (to Michael Ferrara).}
\footnotetext[5]{Department of Mathematics, University of Illinois Urbana-Champaign, Urbana, IL 60801; {\tt reinige1@illinois.edu}.}
\footnotetext[6]{This author's research was partially supported by the National Security Agency (NSA) via grant H98230-13-1-0226. This author's contribution was completed in part while he was a long-term visitor at the Institute for Mathematics and its Applications. He is grateful to the IMA for its support and for fostering such a vibrant research community.}
\begin{abstract}
A sequence of nonnegative integers $\pi =(d_1,d_2,...,d_n)$ is {\it graphic} if there is a (simple) graph $G$ of order $n$ having degree sequence $\pi$.  In this case, $G$ is said to {\it realize} or be a {\it realization of} $\pi$.  Given a graph $H$, a graphic sequence $\pi$ is \textit{potentially $H$-graphic} if there is some realization of $\pi$ that contains $H$ as a subgraph.  

In this paper, we consider a degree sequence analogue to classical graph Ramsey numbers.  For graphs $H_1$ and $H_2$, the \textit{potential-Ramsey number} $\rpot(H_1,H_2)$ is the minimum integer $N$ such that for any $N$-term graphic sequence $\pi$, either $\pi$ is potentially $H_1$-graphic or the complementary sequence $\overline{\pi}=(N-1-d_N,\dots, N-1-d_1)$
is potentially $H_2$-graphic.   

We prove that if $s\ge 2$ is an integer and $T_t$ is a tree of order $t> 7(s-2)$, then $$\rpot(K_s, T_t) = t+s-2.$$ This result, which is best possible up to the bound on $t$, is a degree sequence analogue to a classical 1977 result of Chv\'{a}tal on the graph Ramsey number of trees vs.\ cliques.  To obtain this theorem, we prove a sharp condition that ensures an arbitrary graph packs with a forest, which is likely to be of independent interest.  

\end{abstract}
\section{Introduction}

A sequence of nonnegative integers $\pi =(d_1,d_2,\dotsc,d_n)$ is {\it graphic} if there is a (simple) graph $G$ of order $n$ having degree sequence $\pi$.  In this case, $G$ is said to {\it realize} or be a {\it realization of} $\pi$, and we will write $\pi=\pi(G)$.  Unless otherwise stated, all sequences in this paper are assumed to be nonincreasing.  

There are a number of theorems characterizing graphic sequences, including classical results by Havel~\cite{hav} and Hakimi~\cite{hak} and an independent characterization by Erd\H{o}s and Gallai~\cite{EG}.  However, a given graphic sequence may have a diverse family of nonisomorphic realizations; as such it has become of recent interest to determine when realizations of a given graphic sequence have various properties. As suggested by A.R.~Rao in~\cite{Ra79}, such problems can be broadly classified into two types, the first described as ``forcible'' problems and the second as ``potential'' problems.  In a forcible degree sequence problem, a specified graph property must exist in every realization of the degree sequence $\pi$, while in a potential degree sequence problem, the desired property must be found in at least one realization of $\pi$. 

 Results on forcible degree sequences are often stated as traditional problems in structural or extremal graph theory, where a necessary and/or sufficient condition is given in terms of the degrees of the vertices (or equivalently the number of edges) of a given graph (e.g.~Dirac's theorem on hamiltonian graphs).  
 
A number of degree sequence analogues to classical problems in extremal graph theory appear throughout the literature, including potentially graphic sequence variants of Hadwiger's Conjecture~\cite{CO,DM,RoSo10}, extremal graph packing theorems~\cite{BFHJKW, DFJS}, the Erd\H{o}s-S\'{o}s Conjecture~\cite{LiYin09}, and the Tur\'{a}n Problem~\cite{EJL, FLMW}.  

\subsection{Potential-Ramsey Numbers}

Given graphs $H_1$ and $H_2$, the \textit{Ramsey number} $r(H_1,H_2)$ is the minimum positive integer $N$ such that every red/blue coloring of the edges of the complete graph $K_N$ yields either a copy of $H_1$ in red or a copy of $H_2$ in blue.  In~\cite{BFHJ}, the authors introduced the following potential degree sequence version of the graph Ramsey numbers.  

Given a graph $H$, a graphic sequence $\pi=(d_1,\dots,d_n)$ is \textit{potentially $H$-graphic} if there is some realization of $\pi$ that contains $H$ as a subgraph.  For graphs $H_1$ and $H_2$, the \textit{potential-Ramsey number} $\rpot(H_1,H_2)$ is the minimum integer $N$ such that for any $N$-term graphic sequence $\pi$, either $\pi$ is potentially $H_1$-graphic or the complementary sequence $$\overline{\pi}=(\overline{d}_1,\dots,\overline{d}_N)=(N-1-d_N,\dots, N-1-d_1)$$
is potentially $H_2$-graphic.   

In traditional Ramsey theory, the enemy gives us a coloring of the edges of $K_N$, and we must be sure that there is some copy of $H_1$ or $H_2$ in the appropriate color, regardless of the enemy's choice.  When considering the potential-Ramsey problem, we are afforded the additional flexibility of being able to replace the enemy's red subgraph with any graph having the same degree sequence.  Consequently, it is immediate that for any $H_1$ and $H_2$, $$\rpot(H_1,H_2)\le r(H_1,H_2).$$ 

While this bound is sharp for certain pairs of graphs, for example $H_1=P_n$ and $H_2=P_m$ (see~\cite{BFHJ}), in general that is far from the case.  For instance, determining $r(K_k,K_k)$ is generally considered to be one of the most difficult open problems in combinatorics, if not all of mathematics.  For instance, the best known asymptotic lower bound, due to Spencer \cite{S}, states that $$r(K_k,K_k)\ge (1+o(1))\frac{\sqrt{2}}{e}k2^{\frac{k}{2}}.$$  As a contrast, the following theorem appears in \cite{BFHJ}.

\begin{theorem}[Busch, Ferrara, Hartke and Jacobson~\cite{BFHJ}]\label{theorem:K_nK_t}\label{theorem:ramsey_clique}
For $n\ge t\ge 3$, $$\rpot(K_n,K_t) = 2n+t-4$$ except when $n=t=3$, in which case $\rpot(K_3,K_3)=6$.  
\end{theorem}

Despite this theorem, which implies that the potential-Ramsey numbers are at worst a ``small'' linear function of the order of the target graphs, it remains a challenging problem to determine a meaningful bound on $\rpot(H_1,H_2)$ for  arbitrary $H_1$ and $H_2$.  In addition to its connections to classical Ramsey theory, the problem of determining $\rpot(H_1,H_2)$ also contributes to the robust body of work on subgraph inclusion problems in the context of degree sequences (c.f.~\cite{ChenLiYin,DM,EJL,FLMW,HS,LS}). 

\section{Potential-Ramsey Numbers for Trees vs.\ Cliques} In this paper, we are motivated by the following classical result of Chv\'{a}tal~\cite{Chv}, which is one of the relatively few exact results for graph Ramsey numbers that applies to a broad collection of target graphs. 

\begin{theorem}
For any positive integers $s$ and $t$ and any tree $T_t$ of order $t$, $$r(K_s,T_t)=(s-1)(t-1)+1.$$
\end{theorem}

This elegant result states that the Ramsey number $r(K_s,T_t)$ depends only on $s$ and $t$ and not $T_t$ itself.  Our goal in this paper is to investigate the existence of a similarly general result for $\rpot(K_s, T_t)$.  For $s\ge 2$, let $G=K_{s-2}\vee \overline{K_{t-1}}$ and note that (a) $G$ is the unique realization of its degree sequence, and (b) $G$ does not contain $K_s$ and $\overline{G}$ does not contain any graph $H$ of order $t$ without isolated vertices.  This implies the following general lower bound on the potential Ramsey number.  

\begin{prop}\label{prop:LBtreevsnoisol}
If $H$ is any graph of order $t$ without isolated vertices, then $\rpot(K_s,H)\geq t+s-2$.  In particular, for any tree $T_t$ of order $t$, $\rpot(K_s,T_t)\ge t+s-2$.  
\end{prop}  

The following theorems 
show that for fixed $t$, different choices of $T_t$ may have different potential-Ramsey numbers. 

\begin{theorem}[Busch, Ferrara, Hartke and Jacobson~\cite{BFHJ}]\label{theorem:pathclique}
For $t\ge 6$ and $s\ge 3$, 
$$\rpot(K_s, P_t)=\left\{\begin{array}{ll} 
                            2s - 2 + \left\lfloor\frac{t}{3}\right\rfloor, & \text{ if $s > \left\lfloor \frac{2t}{3} \right\rfloor$}, \medskip\\
                            t+s-2, & \text{ otherwise}.\\
                         \end{array}\right.$$   

\end{theorem}


\begin{theorem}\label{theorem:potram_clique_star}
For $s,t\geq4$, 
$$\rpot(K_s, K_{1,t-1}) = \left\{\begin{array}{ll}
                            2s, & \text{ if } t< s+2, \medskip\\
                            t+s-2, & \text{ otherwise}.\\
                         \end{array}\right.$$   

\end{theorem}

We will prove Theorem~\ref{theorem:potram_clique_star} in Section~\ref{section:proofs}.

There is one notable common feature between these results; namely that the potential-Ramsey number matches the bound given in Proposition~\ref{prop:LBtreevsnoisol} when $t$ is large enough relative to $s$.  This suggests the question: does
there exist a function $f(s)$ such that if $t\ge f(s)$, then for any tree $T_t$ of order $t$, $$\rpot(K_s,T_t)=t+s-2?$$  
Our next result answers this question in the affirmative. Let $\ell(T)$ denote the number of leaves of a tree $T$.  

\begin{theorem}\label{theorem:tree_clique}
Let $T_t$ be a tree of order $t$ and let $s\ge 2$ be an integer. If either 
\begin{itemize}
\item[(i)] $\ell(T_t)\geq s+1$ or
\item[(ii)] $t> 7(s-2)$,
\end{itemize}
\noindent then $\rpot(K_s, T_t)=t+s-2$. 
\end{theorem}

We believe that the coefficient of $7$ in part (ii) of the theorem could be improved, and believe it is feasible that $\rpot(K_s,T_t)=t+s-2$ whenever $t\gtrsim \frac{3}{2}s$, as suggested by Theorem~\ref{theorem:pathclique}, although we pose no formal conjecture here.  

To obtain Theorem~\ref{theorem:tree_clique}, we give a sharp condition that ensures an arbitrary graph and a forest pack.  This result, which we prove in Section~\ref{section:packing}, is likely of independent interest.  In Section~\ref{section:lemmas} we present some technical lemmas, and in Section~\ref{section:proofs} we prove both Theorem~\ref{theorem:potram_clique_star} and Theorem~\ref{theorem:tree_clique}. 

\section{Packing Forests and Arbitrary Graphs}\label{section:packing}

For two graphs $G$ and $H$ where $|V(G)|\geq |V(H)|$, we say that $G$ and $H$ \emph{pack} if there is an injective function $f:V(H)\to V(G)$ such that for any $xy\in E(H)$, $f(x)f(y)\notin E(G)$. In other words, $G$ and $H$ pack if we can embed both of them into $K_{|V(G)|}$ without any edges overlapping. One of the most notable results about packing is the following theorem due to Sauer and Spencer.

\begin{theorem}[Sauer and Spencer~\cite{SS}]\label{thm:sauerspencer}
Let $G$ and $H$ be graphs of order $n$.  If $2\Delta(G)\Delta(H)<n$, then $G$ and $H$ pack.
\end{theorem}

While Theorem \ref{thm:sauerspencer} is tight (see~\cite{KK} for the characterization of the extremal graphs), the condition is not necessary to ensure that $G$ and $H$ pack. In particular, even if $H$ has large maximum degree, if $H$ is sparse enough, then it may still pack with $G$. In light of this, we provide the following, which is in many cases an improvement to the Sauer-Spencer theorem provided one graph happens to be a forest.

\begin{theorem}\label{thm:sauerspencerforest}
Let $F$ be a forest and $G$ be a graph both on $n$ vertices, and let $\comp(F)$ be the number of components of $F$ that contain at least one edge and $\ell(F)$ the number of vertices of $F$ with degree 1. If
\[
3\Delta(G)+\ell(F)-2\comp(F)< n,
\]
then $F$ and $G$ pack.  
\end{theorem}
\begin{proof}
Let $e\in E(F)$ and let $F'=F\setminus\{e\}$. Notice that $\ell(F')\leq\ell(F)+2$, but if $\ell(F')>\ell(F)$, then $\comp(F')>\comp(F)$. Further, if $\comp(F')<\comp(F)$, then $e$ was an isolated edge, so $\ell(F')\leq\ell(F)-2$. Therefore, $\ell(F')-2\comp(F')\leq\ell(F)-2\comp(F)$. 

Consequently, if $3\Delta(G)+\ell(F)-2\comp(F)< n$, then any subgraph of $F$ attained by deleting edges also satisfies this inequality. Therefore we may suppose, for the sake of contradiction, that $F$ is minimal in the sense that for any edge $e\in E(F)$, $F\setminus\{e\}$ packs with $G$.  Going forward, we follow the proof of the Sauer-Spencer theorem from \cite{KK} and improve the bounds at certain steps with the knowledge that $F$ is a forest.

For an embedding $f:V(G)\to V(F)$, we call an edge $uv\in E(F)$ a \emph{conflicting edge} if there is some edge $xy\in E(G)$ such that $f(x)=u$ and $f(y)=v$. As $G$ and $F$ do not pack, any embedding must have at least one conflicting edge. Notice that by the minimality of $F$, there is a packing $f$ of $G$ with $F\setminus\{uv\}$, and since $G$ does not pack with $F$, inserting the edge $uv$ must create a conflicting edge. Therefore, for any $uv\in E(F)$, there is an embedding in which $uv$ is the only conflicting edge.

For an embedding $f:V(G)\to V(F)$ and $u,v\in V(F)$, a \emph{$(u,v;G,F)_f$-link} is a triple $(u,w,v)$ such that $uw$ is an edge in the embedding of $G$ and $wv$ is an edge of $F$. 
Similarly, a \emph{$(u,v;F,G)_f$-link} is a triple $(u,w,v)$ such that $uw$ is an edge in $F$ and $wv$ is an edge in the embedding of $G$. 
Let $f$ be an embedding such that $ux$ is the only conflicting edge, and let $v\in V(F)\setminus\{x\}$.  
We claim that there is either a $(u,v;G,F)_f$-link or a $(u,v;F,G)_f$-link. 
As $ux$ is both an edge of $F$ and an edge in the embedding of $G$, $xv$ is not an edge of $F$ and also not an edge in the embedding of $G$.  
Supposing that $f(u')=u$ and $f(v')=v$, let $f':V(G)\to V(F)$ be defined as $f'(u')=v$, $f'(v')=u$, and $f'(y)=f(y)$ for all $y\notin\{u',v'\}$. 
As $F$ and $G$ do not pack, there must be a conflicting edge under $f'$. 
This conflicting edge cannot be incident to $x$, but must be incident to either $u$ or $v$.
If the conflicting edge is $uw$, then $(u,w,v)$ is a $(u,v;F,G)_f$-link; if the conflicting edge is instead $vw$, then $(u,w,v)$ is a $(u,v;G,F)_f$-link. 

Let $u\in V(F)$ be a non-isolated vertex and let $f:V(G)\to V(F)$ be an embedding such that $ux$ is the only conflicting edge for some $x$.  
Define 
\begin{align*}
V_1 &=\{v\in V(F)\st\text{there is a $(u,v;F,G)_f$-link}\},\\
V_2 &=\{v\in V(F)\st\text{there is a $(u,v;G,F)_f$-link}\}.
\end{align*}
Since $u$ is incident to $\deg_F(u)$ edges of $F$ and each $w\in N_F(u)$ is incident to at most $\Delta(G)$ edges of the embedding of $G$, we have 
\[
|V_1|\leq\deg_F(u)\Delta(G).
\]
Similarly, with $u'\in V(G)$ such that $f(u')=u$, 
\begin{align*}
 |V_2| &\leq \sum_{w'\in N_G(u')} \deg_F(f(w')) \\
 &= 2\deg_G(u') + \sum_{w'\in N_G(u')} (\deg_F(f(w'))-2) \\
 &\leq 2\Delta(G) + \sum_{\substack{w\in V(F):\\ \deg_F(w)\geq2}} (\deg_F(w)-2) \\
 &= 2\Delta(G) + \ell(F)-2\comp(F),
\end{align*}
where the last equality is justified by the fact that for any tree $T$ with at least two vertices, 
\[
\ell(T)=2+\sum_{\substack{w\in V(T):\\ \deg_T(w)\geq 2}}(\deg_T(w)-2).
\]
Since every $v\in V(F)\setminus\{x\}$ is an element of either $V_1$ or $V_2$ and $u\in V_1\cap V_2$,
\[
n\leq |V_1| + |V_2| \leq \deg_F(u)\Delta(G) + 2\Delta(G)+\ell(F)-2\comp(F).
\]
Since this holds for every nonisolated vertex $u$ of $F$, it holds for a leaf of $F$. Therefore, by choosing $u$ to be a leaf,
\[
n\leq 3\Delta(G)+\ell(F)-2\comp(F),
\]
contradicting the hypothesis of the theorem.
\end{proof}

As we are interested in attaining Chv\'{a}tal-type results for the potential-Ramsey number, we will make use of Theorem \ref{thm:sauerspencerforest} when $F$ is a tree.

\begin{cor}\label{cor:sauerspencertree}
Let $G$ and $T$ be graphs of order $n$ where $T$ is a tree. If $3\Delta(G)+\ell(T)-2< n$, then $G$ and $T$ pack.
\end{cor}

The tightness of Corollary~\ref{cor:sauerspencertree}, and hence of Theorem~\ref{thm:sauerspencerforest}, is seen by letting $n$ be even, $G={n\over 2}K_2$ and $T=K_{1,n-1}$. In this case, $\Delta(G)=1$ and $\ell(T)=n-1$, so $3\Delta(G)+\ell(T)-2=n$; however, $G$ and $T$ do not pack. The following proposition shows the asymptotic tightness of Corollary~\ref{cor:sauerspencertree} for any $\Delta(G)$.

\begin{prop}
For any $\epsilon>0$ and $r\in\Z^+$, there is a graph $G$ with $\Delta(G)=r$ and a tree $T$, each with $n$ vertices, such that $3\Delta(G)+\ell(T)-2< (1+\epsilon)n$, but $G$ and $T$ do not pack.
\end{prop}

\begin{proof}
For $m,r\geq 2$ and $n=mr$, let $G=mK_r$ and let $T$ be any tree of order $n$ that is a subdivision of $K_{1,(m-1)r+1}$. It is readily seen that $G$ and $T$ do not pack for any $m$ and $r$. Given $\epsilon>0$, choose $m$ such that $m> 2/\epsilon$. Thus,
\begin{align*}
3\Delta(G)+\ell(T)-2 &= 3(r-1)+(m-1)r+1-2 \\
&= 2r+rm-4=\left(1+{2\over m}\right)n-4\\
&< (1+\epsilon)n.\qedhere
\end{align*}
\end{proof}

\section{Technical Lemmas}\label{section:lemmas}

We begin with the following useful result of Yin and Li.

\begin{theorem}[Yin and Li~\cite{YiLi05}]\label{thm:cliquegraph}
Let $\pi=(d_1,\dots,d_n)$ be a graphic sequence and $s\geq 1$ be an integer.
\begin{itemize}
\item[(i)] If $d_s\geq s-1$ and $d_{2s}\geq s-2$, then $\pi$ is potentially $K_s$-graphic.
\item[(ii)] If $d_s\geq s-1$ and $d_i\geq 2s-2-i$ for $1\leq i\leq s-1$, then $\pi$ is potentially $K_s$-graphic.
\end{itemize}
\end{theorem}

The following lemma from~\cite{GJL} is an extension of a corresponding result of Rao~\cite{Ra79} for cliques.

\begin{lemma}[Gould, Jacobson and Lehel~\cite{GJL}]\label{lemma:gould}
If $\pi$ is potentially $H$ graphic, then there is a realization $G$ of $\pi$
containing $H$ as a subgraph on the $|V(H)|$ vertices of largest degree in $G$.
\end{lemma}

We next give a simple sufficient condition for a graphic sequence to have a realization containing any tree of order $t$.   

\begin{lemma}\label{lemma:tree_pot}
Let $T$ be a tree of order $t$ and let $\pi=(d_1,\dots,d_n)$ be a graphic sequence with $n\geq t$. If $d_{t-1}\geq t-1$, then $\pi$ is potentially $T$-graphic.
\end{lemma}

\begin{proof}
We proceed by induction on $t$. If $t=2$, the claim follows immediately, so suppose that $t>2$ and let $\pi=(d_1,\dots,d_n)$ be a graphic sequence with $d_{t-1}\geq t-1$. 

Let $T$ be any tree of order $t$, $v$ be a leaf of $T$, and $T'=T\setminus\{v\}$. Thus, $T'$ is a tree of order $t-1$.
As $\pi$ is a non-increasing sequence, $d_{t-2}\geq t-1>t-2$, so by the induction hypothesis, $\pi$ is potentially $T'$-graphic.  By Lemma~\ref{lemma:gould}, there is a realization $G$ of $\pi$ such that $T'$ is a subgraph of $G[v_1,\dots,v_{t-1}]$. 
As $\deg_G(v_i)\geq t-1$ for all $i\leq t-1$, $|N_G(v_i)\cap\{v_t,\dots,v_n\}|\geq 1$ for all $i\leq t-1$. 
Hence we may attach a leaf to any vertex of $T'$ to attain a copy of $T$, so $\pi$ is potentially $T$-graphic.
\end{proof}

Finally, we combine Lemma~\ref{lemma:tree_pot} and Theorem~\ref{thm:cliquegraph} to obtain the following.  

\begin{prop}\label{prop:degrees}
Let $T$ be a tree on $t$ vertices, let $s\leq t-2$, and set $n=t+s-2$.  
If $\pi=(d_1,\dots,d_n)$ is a graphic sequence such that 
 $\pi$ is not potentially $T$-graphic and
 $\overline{\pi}$ is not potentially $K_s$-graphic, 
then $d_{t-s-1}\geq t$ and $d_t\geq t-s+1$.
\end{prop}

\begin{proof}
If $\pi$ is not potentially $T$-graphic, then $d_{t-1}\leq t-2$. Therefore, $\overline{d}_s=n-1-d_{t-1}\geq s-1$. By Theorem \ref{thm:cliquegraph}, if $\overline{d}_{2s}\geq s-2$, then $\overline{\pi}$ is potentially $K_s$-graphic. As this is not the case, $\overline{d}_{2s}\leq s-3$. Hence, $d_{t-s-1}=n-1-\overline{d}_{2s}\geq t$.

Furthermore, since $\overline{d}_s\geq s-1$, we see that $\overline{d}_i\leq 2s-3-i$ for some $1\leq i\leq s-2$ (else $\overline{\pi}$ is potentially $K_s$-graphic by Theorem~\ref{thm:cliquegraph}).  Therefore 
\[d_t=n-1-\overline{d}_{s-1}\geq n-1-\overline{d}_i \geq n-1-(2s-4)=t-s+1.\qedhere\]
\end{proof}

\section{Proof of Theorems \ref{theorem:potram_clique_star} and \ref{theorem:tree_clique}}\label{section:proofs}

We begin by proving Theorem \ref{theorem:potram_clique_star}.

\begin{proof}[Proof of Theorem~\ref{theorem:potram_clique_star}]
When $s\leq t-2$, the lower bound is established by Proposition~\ref{prop:LBtreevsnoisol}.  When $s\geq t-2$, the lower bound is established by considering the graphic sequence $\pi=(d_1,\dots,d_{2s-1})$ where $d_i=2$ for all $i$. 
As $t\geq 4$, $\pi$ is not potentially $K_{1,t-1}$-graphic. 
Further, in any realization of $\pi$, the largest independent set has size at most $s-1$, so $\overline{\pi}$ is not potentially $K_s$-graphic.

For the upper bound, consider any graphic sequence $\pi=(d_1,\dotsc,d_n)$ with $n=s+\max\{s,t-2\}$.  If $d_1\geq t-1$, then $\pi$ is potentially (in fact forcibly) $K_{1,t-1}$-graphic.  So suppose $d_1\leq t-2$.  Thus, we have $\overline{d}_n=n-1-d_1\geq n-1-(t-2)=n-t+1\geq s-1$; since also $n\geq2s$, Theorem~\ref{thm:cliquegraph} implies that $\overline{\pi}$ is potentially $K_s$-graphic.
\end{proof}

We are now ready to prove Theorem~\ref{theorem:tree_clique}, the main result of this paper.  

\begin{proof}[Proof of Theorem~\ref{theorem:tree_clique}.]
Let $\pi=(d_1,\dots,d_n)$ be a graphic sequence with $n=t+s-2$ such that $\pi$ is not potentially $T$-graphic and $\overline{\pi}$ is not potentially $K_s$-graphic.

(i) Suppose that $\ell(T)\geq s+1$ and let $S$ be any set of $s+1$ leaves of $T_t$. Let $T'=T\setminus S$, so $T'$ is a tree of order $t-s-1$. By Proposition~\ref{prop:degrees}, $d_{t-s-2}\geq t>t-s-2$, so $\pi$ is potentially $T'$-graphic. Let $G$ be a realization of $\pi$ with $T'$ as a subgraph of $G[v_1,\dots,v_{t-s-1}]$. Because $\deg_G(v_i)\geq t$ for all $i\leq t-s-1$, $|N_G(v_i)\cap\{v_{t-s},\dots,v_{t-s-1}\}|\geq s+2$ for all $i\leq t-s-1$. Hence, we can reattach the vertices in $S$ to attain a copy of $T$, contradicting the fact that $\pi$ is not potentially $T$-graphic.

(ii) Suppose that $\ell(T)\leq s$ and let $G$ be any realization of $\pi$ and define $G_t=G[v_1,\dots,v_t]$. As $d_t\geq t-s+1$ and $|V(G)|=t+s-2$, $\delta(G_t)\geq d_t-(s-2)=t-2s+3$, so $\Delta(\overline{G_t})=t-1-\delta(G_t)\leq 2s-4$. Because $t>7(s-2)$,
\[
3\Delta(\overline{G_t})+\ell(T)-2\leq 3(2s-4)+s-2=7s-14< t.
\]
By applying Corollary~\ref{cor:sauerspencertree}, we see that $T$ and $\overline{G_t}$ pack, or in other words, $T$ is a subgraph of $G_t$, again contradicting the fact that $\pi$ is not potentially $T$-graphic.
\end{proof}


\begin{thebibliography}{99}

\bibitem{BFHJ} A.\ Busch, M.\ Ferrara, S.\ Hartke and M.\ Jacobson, Ramsey-type numbers for degree sequences, \textit{Graphs Comb.} \textbf{30} (2014), 847--859.

\bibitem{BFHJKW} A.\ Busch, M.\ Ferrara, M.\ Jacobson, H.\ Kaul, S.\ Hartke and D.\ West, Packing of graphic $n$-tuples, {\it J. Graph Theory} \textbf{70} (2012), 29--39.

\bibitem{ChenLiYin} G. Chen, J.S.\ Li, J.H.\ Yin, A variation of a classical 
Tur\'{a}n-type extremal problem, {\it European J. Combin.} 
{\bf 25} (2004) 989-1002.

\bibitem{CO} G. Chen and K.\ Ota, Hadwiger's conjecture for degree sequences, to appear in {\it J. Combin. Theory. Ser. B.}

\bibitem{Chv} V.\ Chv\'{a}tal, Tree-complete graph Ramsey numbers, {\it J. Graph Theory} {\bf 1} (1977), 93.  

\bibitem{DFJS} J.\ Diemunsch, M.\ Ferrara, S.\ Jahanbekam and J.\ Shook, Extremal Theorems for degree sequence packing and the 2-color discrete tomography problem, submitted.

\bibitem{DM} Z.\ Dvo\v{r}\'{a}k and B.\ Mohar, Chromatic number and complete graph substructures for degree sequences, \textit{Combinatorica} \textbf{33} (2013) 513--529.

\bibitem{EG} P.\ Erd\H os and T.\ Gallai, Graphs with prescribed degrees, \textit{Matematiki Lapor} \textbf{11} (1960), 264--274 (in Hungarian).

\bibitem{EJL}P.\ Erd\H{o}s, M.\ Jacobson and J.\ Lehel, Graphs realizing
the same degree sequence and their respective clique numbers, \textit{Graph Theory, Combinatorics and Applications} (eds. Alavi, Chartrand, Oellerman and Schwenk), Vol. I, 1991, 439--449.

\bibitem{FLMW} M.\ Ferrara, T.\ LeSaulnier, C.\ Moffatt and P.\ Wenger, On the sum necessary to ensure that a degree sequence is potentially $H$-graphic, to appear in \textit{Combinatorica}. 

\bibitem{GJL}  R.\ Gould, M.\ Jacobson and J.\ Lehel, Potentially
$G$-graphic degree sequences, {\it Combinatorics, Graph Theory, and
Algorithms} (eds. Alavi, Lick and Schwenk), Vol. I, New York: Wiley
\& Sons, Inc., 1999, 387--400.

\bibitem{hak} S.L.\ Hakimi, On the realizability of a set of integers as
degrees of vertices of a graph, {\it J. SIAM Appl. Math}, {\bf 10} (1962), 496--506.

\bibitem{HaSc78} S.\ Hakimi and E.\ Schmeichel, Graphs and their degree sequences: A survey. In
\textit{Theory and Applications of Graphs}, volume 642 of \textit{Lecture Notes in
Mathematics}, Springer Berlin / Heidelberg, 1978, 225--235.

\bibitem{HS} S.\ Hartke and T.\ Seacrest, Graphic sequences have realizations containing bisections of large degree, {\it J. Graph Theory} \textbf{71} (2012), 386--401.

\bibitem{hav} V.\ Havel,  A remark on the existence of finite graphs (Czech.),
{\it $\check{C}asopis~P\check{e}st.~Mat.$} {\bf 80} (1955), 477--480.

\bibitem{KK} H.\ Kaul and A.\ Kostochka, Extremal graphs for a graph packing theorem of Sauer and Spencer, {\it Combin. Probab. Comput.} {\bf 16} (2007), no. 3, 409--416.

\bibitem{LS}  J.\ Li and Z.\ Song, The smallest degree sum
that yields potentially $P_k$-graphical sequences, {\it J. Graph Theory}
{\bf 29} (1998), 63--72.



\bibitem{Ra79} A.R.\ Rao, The clique number of a graph with a given degree sequence. In \textit{Proceedings of the
Symposium on Graph Theory}, volume 4 of \textit{ISI Lecture
Notes}, Macmillan of India, New Delhi, 1979, 251--267.

\bibitem{SRa81} S.B.\ Rao, A survey of the theory of potentially $P$-graphic and forcibly $P$-graphic degree
sequences. In \textit{Combinatorics and graph theory}, volume 885 of \textit{Lecture Notes
in Math.}, Springer, Berlin, 1981, 417--440.

\bibitem{RoSo10} N.\ Robertson and Z.\ Song, Hadwiger number and chromatic number for near regular degree
sequences, \textit{J. Graph Theory} \textbf{64} (2010), 175--183.

\bibitem{SS} N.\ Sauer and J.\ Spencer, Edge disjoint placement of graphs, 
\textit{J. Combin. Theory Ser. B} \textbf{25} (1978), no. 3, 295–302. 

\bibitem{S} J.\ Spencer, Ramsey's theorem - a new lower bound, {\it J. Combin. Theory, Ser. A} {\bf 18} (1975), 108--115.

\bibitem{YiLi05}
J.H.\ Yin and J.S.\ Li, Two sufficient conditions for a graphic sequence to have a realization with prescribed clique size, \textit{Discrete Math.} \textbf{301} (2005), 218--227.

\bibitem{LiYin09} J.H.\ Yin and J.S.\ Li, A variation of a conjecture due to Erd\H{o}s and S\'{o}s,  \textit{Acta Math. Sin.} {\bf 25} (2009), 795--802.




\end{thebibliography}
\end{document}